\documentclass[preprint,12pt]{elsarticle}

\usepackage{amsmath,amssymb,amsthm,mathtools}
\usepackage{enumitem}
\usepackage{hyperref}
\hypersetup{hidelinks}
\usepackage[nameinlink,noabbrev]{cleveref}
\usepackage{microtype}

\journal{Finite Fields and Their Applications}

\newtheorem{theorem}{Theorem}[section]
\newtheorem{lemma}[theorem]{Lemma}
\newtheorem{corollary}[theorem]{Corollary}

\newtheorem{remark}[theorem]{Remark}
\newtheorem{example}[theorem]{Example}

\newcommand{\F}{\mathbb F}

\begin{document}

\begin{frontmatter}

\title{Sparse Lifts of $q$-Linearized Permutation Polynomials with Explicit Inverses}

\author[addr1]{Yi Li}
\address[addr1]{School of Computer Science, Shanghai Jiao Tong University, Shanghai 200240, China}

\begin{abstract}
We give a higher unit-torus lift from $q$-linearized permutation polynomials over $\F_{q^n}$ to generally nonlinear sparse permutation polynomials over $\F_{q^{ns}}$.  The construction preserves compositional inverses: if a linearized polynomials and its inverse have coefficients $(a_i)$ and $(A_i)$, then the lifted polynomial and its inverse have the same coefficients. As a consequence, we obtain nonlinear binomial lifts whose inverses are again binomials in the two-dimensional case, including an explicit inverse for a binomial family over $\F_{q^4}$. The three-dimensional specialization yields permutation trinomials whose compositional inverses are again trinomials, together with explicit involution criteria and examples.
\end{abstract}


\end{frontmatter}

\section{Introduction}

Let $\F_q$ be a finite field.  A polynomial $f\in\F_q[x]$ is a permutation polynomial if it induces a bijection of $\F_q$.  Once $f$ is a permutation polynomial, its compositional inverse is the unique polynomial of degree at most $q-1$ representing the inverse map on $\F_q$.  The construction of permutation polynomials and the computation of their compositional inverses are closely related but technically different problems.  A proof of permutation behavior may only require a quotient map, a fiberwise bijection, or a character-sum criterion; an explicit inverse must keep track of the inverse map on all fibers.  This distinction becomes especially important for sparse polynomials, where one would like the inverse to remain sparse.

The inverse problem for permutation polynomials over finite fields has been studied through several complementary methods.  Wang obtained explicit inverse formulas for certain permutation polynomials and later for related cyclotomic mapping permutations~\cite{WangInversePP,WangCyclotomicNote}.  Piecewise constructions for inverses of cyclotomic mapping permutation polynomials were studied by Zheng, Yu, Zhang and Pei~\cite{ZhengYuZhangPeiCyclotomic}.  Li, Qu and Wang treated permutation polynomials of the form $x^r h(x^s)$ and described their compositional inverses through the induced maps on a subgroup and its fibers~\cite{LiQuWang2019}.  The AGW criterion has also become an important tool for both constructing permutations and computing inverses; see Akbary, Ghioca and Wang~\cite{AGW}, Niu, Li, Qu and Wang~\cite{NiuLiQuWang2021}, and Yuan~\cite{YuanAGW}.  Yuan's local method gives a systematic way to recover inverses from local inverse data~\cite{YuanLocal}, and Wu and Yuan applied this viewpoint to new classes of permutation polynomials and their inverses~\cite{WuYuanLocal}.  There are also important inverse formulas for linearized and related semilinear permutation polynomials, including work of Tuxanidy and Wang~\cite{TuxanidyWang2014}, Wu and Liu~\cite{WuLiuRevisited}, and Wu~\cite{WuLinearizedOdd}.  Further developments and many references are summarized in Wang's recent survey on compositional inverses of permutation polynomials~\cite{WangSurvey2025}.  Inverses of small-degree permutation polynomials were investigated by Zheng, Wang and Wei~\cite{ZhengWangWeiSmallDegree}.

In this paper, we introduce a higher unit-torus lifting mechanism which starts from a $q$-linearized permutation polynomial over $\F_{q^n}$ and produces a generally nonlinear sparse permutation polynomial over a larger field $\F_{q^{ns}}$.  The main feature is that the compositional inverse is lifted at the same time.  Thus, if the $q$-linearized polynomial has a sparse inverse, then the lifted nonlinear permutation polynomial also has an explicitly sparse inverse.  This gives a direct way to construct nonlinear sparse permutation polynomials whose inverses can still be written with a controlled number of terms.

The main theorem gives a sparse lift from $q$-linearized permutation polynomials to nonlinear permutation polynomials over higher extension fields. The exponents in the lifted polynomial are organized by the norm-one torus of order $\frac{q^n-1}{q-1}$, while the algebraic input is the compositional inverse of a $q$-linearized polynomial over $\mathbb{F}_{q^n}$. The exact exponent construction and proof are given in \cref{sec:setup,sec:main}.  This should be distinguished from general inverse algorithms for $x^rh(x^s)$-type or cyclotomic mapping permutations. Those algorithms are designed for broad applicability, whereas our contribution is a concrete sparse construction: if the underlying $q$-linearized polynomial has a sparse inverse, then the lifted nonlinear permutation polynomial has an explicit inverse obtained by lifting the same sparse terms.

Several concrete consequences follow from the main theorem.  In dimension two, the construction gives nonlinear permutation binomials whose compositional inverses are again binomials.  In particular, it supplies an explicit binomial inverse for a family of permutation binomials appearing in the work of Li, Feng and Wang~\cite{LiFengWangBinomials}.  To the best of our knowledge, few explicit examples are known of permutation binomials whose compositional inverses are again binomials. By contrast, general inverse methods typically produce inverses that are piecewise-defined, cyclotomic, or multi-term. The same two-dimensional specialization also yields explicit criteria for the resulting permutation polynomials to be involutions, together with concrete examples.  In dimension three, the lifting construction involves up to three nonzero terms, with exponents determined by the norm-one torus.  This yields explicit permutation trinomials whose compositional inverses are again trinomials. We also give a criterion for trinomial involutions and explicit nonlinear examples over binary fields.

The paper is organized as follows. \Cref{sec:setup} introduces the norm-one torus and the exponent notation used throughout the paper.  \Cref{sec:main} proves the lifting theorem and a descent theorem.  \Cref{sec:binomial} specializes the construction to the two-dimensional case and gives explicit inverse formulas for permutation binomials whose compositional inverses are also binomials, as well as criteria for involutions. \Cref{sec:trinomial} treats the three-dimensional case via the Dickson matrix of a $q$-linearized trinomial, giving explicit permutation trinomials whose compositional inverses are again trinomials, together with criteria and examples for involutions. The final section summarizes the construction and its possible extensions.

\section{The higher unit-torus setup}\label{sec:setup}

Let $q$ be a prime power, and let $n,s$ be positive integers.  Put
\[
        K=\F_{q^{ns}},\qquad L=\F_{q^n},\qquad N=q^{ns}-1,
\]
and
\[
        M=1+q+\cdots+q^{n-1}=\frac{q^n-1}{q-1}.
\]
Throughout the main construction, we assume
\[
        \gcd(s,M)=1.
\]
Let
\[
        S=\frac{N}{M}.
\]
Choose an integer $h$ satisfying
\[
        hs\equiv 1\pmod M,
\]
and define $T$ modulo $N$ by
\[
        T\equiv hS\pmod N.
\]

\begin{lemma}\label{lem:Tproperties}
The exponent $T$ has the following properties.
\begin{enumerate}[label=\rm(\roman*)]
    \item For every $x\in K^*$, one has $x^T\in \mu_M$.
    \item One has
    \[
        T\equiv q-1\pmod {q^n-1}.
    \]
Consequently, for every $z\in L^*$,
\[
        z^T=z^{q-1}.
    \]
\end{enumerate}
\end{lemma}

\begin{proof}
Since $M T\equiv M hS=hN\equiv0\pmod N$, every $x^T$ has order dividing $M$.  Hence $x^T\in \mu_M$.

For the second assertion, observe that
\[
S=\frac{q^{ns}-1}{M}=(q-1)(1+q^n+q^{2n}+\cdots+q^{n(s-1)}).
\]
Modulo $q^n-1$ we have $q^n\equiv1$, and therefore
\[
        S\equiv s(q-1)\pmod {q^n-1}.
\]
Multiplying by $h$ gives
\[
        T\equiv hS\equiv hs(q-1)\equiv q-1\pmod {q^n-1},
\]
because $hs\equiv1\pmod M$ and $M$ divides $q^n-1$.  The final claim follows from $|L^*|=q^n-1$.
\end{proof}

For $0\le i\le n-1$, define
\[
        e_i=\frac{q^i-1}{q-1}=1+q+\cdots+q^{i-1},
\]
with the convention $e_0=0$.

\begin{lemma}\label{lem:freq}
For all $i,j$, one has
\[
        e_i+q^i e_j\equiv e_{i+j}\pmod M,
\]
where the subscript $i+j$ is taken modulo $n$.
\end{lemma}

\begin{proof}
As integers,
\[
        e_i+q^i e_j=(1+q+\cdots+q^{i-1})+q^i(1+q+\cdots+q^{j-1})=e_{i+j}
\]
when $i+j<n$.  If $i+j\ge n$, then $e_{i+j}=e_n+q^n e_{i+j-n}$.  Since $e_n=M$ and $q^n\equiv1\pmod M$, we get $e_{i+j}\equiv e_{i+j-n}\pmod M$.  This proves the assertion.
\end{proof}

\section{ Lifting linearized permutations and descent}\label{sec:main}

Let
\[
        \mathcal L_A(X)=\sum_{i=0}^{n-1}a_i X^{q^i}\in L[X]
\]
be a $q$-linearized polynomial over $L$.  Assume that $\mathcal L_A$ is a permutation polynomial of $L$, and write its inverse as
\[
        \mathcal L_A^{-1}(X)=\mathcal L_B(X)=\sum_{i=0}^{n-1}A_i X^{q^i}.
\]

\begin{theorem}[Higher unit-torus lift]\label{thm:forward}
Let $r$ be an integer satisfying
\[
        \gcd(r,N)=1,
        \qquad
        r\equiv1\pmod {q^n-1}.
\]
Let $r'$ be the inverse of $r$ modulo $N$.  Define
\[
        F_A(x)=x^r\sum_{i=0}^{n-1}a_i x^{e_iT}
        =\sum_{i=0}^{n-1}a_i x^{r+e_iT}\in K[x]
\]
and
\[
        F_B(x)=x^{r'}\sum_{i=0}^{n-1}A_i x^{e_iT}
        =\sum_{i=0}^{n-1}A_i x^{r'+e_iT}\in K[x].
\]
Then $F_A$ is a permutation polynomial of $K$, and
\[
        F_A^{-1}=F_B.
\]
\end{theorem}

\begin{proof}
It is enough to prove $F_A(F_B(y))=y$ for all $y\in K$.  The identity is clear for $y=0$, so assume $y\in K^*$.  Put
\[
        v=y^T\in\mu_M
\]
and
\[
        R(v)=\sum_{j=0}^{n-1}A_j v^{e_j}\in L.
\]
Then
\[
        F_B(y)=y^{r'}R(v).
\]
Since $rr'\equiv1\pmod N$ and $R(v)\in L$, while $r\equiv1\pmod {q^n-1}$, we get
\[
        F_B(y)^r=yR(v).
\]
Moreover, $r'\equiv1\pmod M$, and by Lemma~\ref{lem:Tproperties},
\[
        F_B(y)^T=y^{r'T}R(v)^T=vR(v)^{q-1}.
\]
Therefore
\[
\begin{aligned}
F_A(F_B(y))
&=F_B(y)^r\sum_{i=0}^{n-1}a_i\bigl(F_B(y)^T\bigr)^{e_i} \\
&=yR(v)\sum_{i=0}^{n-1}a_i\left(vR(v)^{q-1}\right)^{e_i}.
\end{aligned}
\]
Since $(q-1)e_i=q^i-1$, this becomes
\[
        F_A(F_B(y))=y\sum_{i=0}^{n-1}a_i v^{e_i}R(v)^{q^i}.
\]
Now expand
\[
        R(v)^{q^i}=\sum_{j=0}^{n-1}A_j^{q^i}v^{q^i e_j}.
\]
By Lemma~\ref{lem:freq}, the exponent $e_i+q^ie_j$ reduces modulo $M$ to $e_{i+j}$.  Hence
\[
\sum_{i=0}^{n-1}a_i v^{e_i}R(v)^{q^i}
=
\sum_{i,j}a_i A_j^{q^i}v^{e_{i+j}}.
\]
The coefficients in this expression are precisely the coefficients of the $q$-linearized composition
\[
        \mathcal L_A(\mathcal L_B(X)).
\]
Since $\mathcal L_B=\mathcal L_A^{-1}$ over $L$, this composition equals $X$.  Thus the above expression equals $1$, and consequently $F_A(F_B(y))=y$.  The same argument, or bijectivity, gives $F_B(F_A(y))=y$.
\end{proof}

The next result gives a converse aspect of the lifting construction.  It shows that the construction is not merely sufficient: whenever a polynomial map written with the same prescribed set of exponents has an inverse using that same set of exponents, its coefficient data descend to a pair of mutually inverse $q$-linearized polynomials over $L$.

\begin{theorem}[Descent theorem]\label{thm:descent}
Keep the notation and assumptions on $K,L,M,T$ and $r$ as above.  Let
\[
        F_A(x)=x^r\sum_{i=0}^{n-1}a_i x^{e_iT},
        \qquad
        F_B(x)=x^{r'}\sum_{i=0}^{n-1}A_i x^{e_iT},
\]
where $a_i,A_i\in L$ and $rr'\equiv1\pmod N$.  If
\[
        F_A(F_B(x))=x
\]
as a function on $K$, then
\[
        \left(\sum_{i=0}^{n-1}a_iX^{q^i}\right)
        \circ
        \left(\sum_{i=0}^{n-1}A_iX^{q^i}\right)
        =X
\]
as a function on $L$. The coefficient data of the nonlinear inverse pair descend to a pair of mutually inverse $q$-linearized permutation polynomials over $L$.
\end{theorem}

\begin{proof}
As in the proof of Theorem~\ref{thm:forward}, for $v\in\mu_M$ put
\[
        R(v)=\sum_{j=0}^{n-1}A_jv^{e_j}.
\]
The identity $F_A(F_B(y))=y$ for $y\ne0$ implies
\[
        \sum_{i=0}^{n-1}a_i v^{e_i}R(v)^{q^i}=1
        \qquad (v\in\mu_M).
\]
Expanding gives
\[
        \sum_{k=0}^{n-1} C_k v^{e_k}=1,
\]
where $C_k$ is exactly the coefficient of $X^{q^k}$ in the $q$-linearized composition
\[
        \mathcal L_A(\mathcal L_B(X)).
\]
The characters $v^{e_0},\ldots,v^{e_{n-1}}$ are distinct on $\mu_M$, because
\[
        0=e_0<e_1<\cdots<e_{n-1}<e_n=M.
\]
They are linearly independent as functions on $\mu_M$.  Indeed, if
\[
        \sum_{i=0}^{n-1}\lambda_i v^{e_i}=0
        \qquad\text{for all }v\in\mu_M,
\]
then the polynomial
\[
        P(X)=\sum_{i=0}^{n-1}\lambda_iX^{e_i}
\]
has all elements of $\mu_M$ as roots.  Since $\deg P<M$ and $X^M-1$ is separable, this forces $P=0$, and hence all $\lambda_i=0$.  Therefore $C_0=1$ and $C_k=0$ for $1\le k\le n-1$.  Thus the $q$-linearized composition equals $X$ on $L$.
\end{proof}

\begin{remark}
For $n=2$, the construction involves only two prescribed exponents, namely 0 and 1, and therefore produces permutation binomials whose compositional inverses are again binomials. For $n=3$, the prescribed exponents are $0$, $1$, and $q+1$, leading naturally to permutation trinomials whose compositional inverses are again trinomials.
\end{remark}

\section{The binomial case}\label{sec:binomial}

We now specialize to $n=2$.  Then
\[
        M=q+1,
        \qquad
        e_0=0,
e_1=1.
\]
Let
\[
        L_{a,b}(X)=aX+bX^q\in \F_{q^2}[X].
\]
The Dickson matrix is
\[
        \begin{pmatrix} a&b\\ b^q&a^q\end{pmatrix},
\]
and
\[
        D=a^{q+1}-b^{q+1}.
\]
The polynomial $L_{a,b}$ permutes $\F_{q^2}$ if and only if $D\ne0$, in which case
\[
        L_{a,b}^{-1}(X)=\frac{a^q}{D}X-\frac{b}{D}X^q.
\]

In this case, Theorem~\ref{thm:forward} yields the following corollary:

\begin{corollary}\label{cor:binomial}
Let $K=\F_{q^{2s}}$, assume $\gcd(s,q+1)=1$, and define $T$ as in Section~\ref{sec:setup}.  Let $r$ satisfy
\[
        \gcd(r,q^{2s}-1)=1,
        \qquad
        r\equiv1\pmod {q^2-1},
\]
and let $r'$ be its inverse modulo $q^{2s}-1$.  If $a,b\in\F_{q^2}^*$ satisfy
\[
        D=a^{q+1}-b^{q+1}\ne0,
\]
then
\[
        F(x)=a x^r+b x^{r+T}
\]
permutes $K$, and
\[
        F^{-1}(x)=\frac{a^q}{D}x^{r'}-\frac{b}{D}x^{r'+T}.
\]
In characteristic two, this reads
\[
        F^{-1}(x)=\frac{a^q}{a^{q+1}+b^{q+1}}x^{r'}+\frac{b}{a^{q+1}+b^{q+1}}x^{r'+T}.
\]
\end{corollary}

Here is an example. 

\begin{example}\label{ex:q4binomial}
Assume that $q$ is even and let $K=\F_{q^4}$.  The following permutation binomial family was studied in~\cite{LiFengWangBinomials}; We now give an explicit formula for its compositional inverse.
  Here $s=2$ and
\[
        M=q+1,
        \qquad
        N=q^4-1,
        \qquad
        S=\frac{q^4-1}{q+1}=(q-1)(q^2+1).
\]
Since $q+1$ is odd, the inverse of $2$ modulo $q+1$ is $(q+2)/2$.  Hence
\[
        T\equiv \frac{q+2}{2}S\pmod {q^4-1}.
\]
A direct calculation gives
\[
        T+1=\frac{q^4+q^3-q^2+q}{2}=:E.
\]
Taking $r=1$, $b=1$, and $a\in\F_{q^2}^*$ with $a^{q+1}\ne1$, Corollary~\ref{cor:binomial} gives
\[
        f_a(x)=x^E+a x
\]
as a permutation polynomial of $\F_{q^4}$, with compositional inverse
\[
        f_a^{-1}(x)=\frac{a^q}{1+a^{q+1}}x+\frac{1}{1+a^{q+1}}x^E.
\]
\end{example}

Furthermore, we give an explicit criterion for the resulting permutation polynomials to be involutions.

\begin{corollary}\label{cor:involution}
Keep the notation of Theorem~\ref{cor:binomial}.  Suppose in addition that
\[
        r^2\equiv1\pmod {q^{2s}-1}
\]
and that the linear binomial $L_{a,b}(X)=aX+bX^q$ is an involution of $\F_{q^2}$, i.e.
\[
        a^2+b^{q+1}=1,
        \qquad
        b(a+a^q)=0.
\]
Then
\[
        F(x)=a x^r+b x^{r+T}
\]
is an involution of $K$.
In particular, if $q$ is even, $a\in\F_q$, $a\ne0,1$, and $b\in\F_{q^2}^*$ satisfies
\[
        b^{q+1}=1+a^2,
\]
then $F$ is an involution whenever $r^2\equiv1\pmod {q^{2s}-1}$.
\end{corollary}

\begin{proof}
If $r^2\equiv1\pmod {q^{2s}-1}$, then $r'=r$.  Applying Theorem~\ref{thm:forward} to the equality $L_{a,b}^{-1}=L_{a,b}$ gives $F^{-1}=F$.
\end{proof}

The following example gives a permutation binomial that is an involution.

\begin{example}\label{ex:binomial-involution}
Let $q=4$, $s=3$, and $K=\F_{q^{2s}}=\F_{2^{12}}$.  Then
\[
        N=q^{2s}-1=4095,\qquad M=q+1=5,
\]
and
\[
        S=N/M=819.
\]
Since $2s=6\equiv1\pmod 5$, one may take $h=2$, and hence
\[
        T=2S=1638.
\]
Let $\beta$ be a primitive element of $\F_{16}$ and set $\omega=\beta^5\in\F_4^*$.  Then $\omega^2+\omega+1=0$.  Take
\[
        a=\omega,
        \qquad
        b=\beta,
        \qquad
        r=181.
\]
One has
\[
        r\equiv1\pmod {15},\qquad r^2\equiv1\pmod {4095},
\]
and
\[
        b^{q+1}=\beta^5=\omega=1+\omega^2=1+a^2.
\]
Thus Corollary~\ref{cor:involution} gives the nonlinear binomial involution
\[
        I(x)=\omega x^{181}+\beta x^{1819}\in\F_{2^{12}}[x].
\]
Its compositional inverse is itself:
\[
        I^{-1}(x)=I(x)=\omega x^{181}+\beta x^{1819}.
\]
The exponents $181$ and $1819$ are not powers of $2$ modulo $4095$, so this is not a linearized binomial.
\end{example}

\section{The trinomial case}\label{sec:trinomial}

We now specialize to $n=3$.  Then
\[
        M=q^2+q+1,
        \qquad
        e_0=0,
e_1=1,
e_2=q+1.
\]
Let
\[
        \mathcal L(X)=aX+bX^q+cX^{q^2}\in\F_{q^3}[X].
\]
Its Dickson matrix is
\[
        \mathcal D(a,b,c)=
        \begin{pmatrix}
        a & b & c\\
        c^q & a^q & b^q\\
        b^{q^2} & c^{q^2} & a^{q^2}
        \end{pmatrix}.
\]
Put
\[
\Delta=
 a^{1+q+q^2}+b^{1+q+q^2}+c^{1+q+q^2}
-ab^q c^{q^2}-bc^q a^{q^2}-ca^q b^{q^2}.
\]
Then $\mathcal L$ permutes $\F_{q^3}$ if and only if $\Delta\ne0$.

When $\Delta\ne0$, the inverse has the form
\[
        \mathcal L^{-1}(X)=AX+BX^q+CX^{q^2},
\]
where
\[
        A=\frac{a^{q+q^2}-b^qc^{q^2}}{\Delta},
\]
\[
        B=\frac{c^{q^2+1}-a^{q^2}b}{\Delta},
\]
and
\[
        C=\frac{b^{q+1}-ca^q}{\Delta}.
\]
These are the entries in the first row of $\mathcal D(a,b,c)^{-1}$.

In this case, Theorem~\ref{thm:forward} yields the following corollary:

\begin{corollary}\label{cor:trinomial}
Let $K=\F_{q^{3s}}$ and assume $\gcd(s,q^2+q+1)=1$.  Define $T$ as in Section~\ref{sec:setup}.  Let $r$ satisfy
\[
        \gcd(r,q^{3s}-1)=1,
        \qquad
        r\equiv1\pmod {q^3-1},
\]
and let $r'$ be the inverse of $r$ modulo $q^{3s}-1$.  If $a,b,c\in\F_{q^3}$ satisfy $\Delta\ne0$, then
\[
        F(x)=a x^r+b x^{r+T}+c x^{r+(q+1)T}
\]
permutes $K$, and
\[
        F^{-1}(x)=A x^{r'}+B x^{r'+T}+C x^{r'+(q+1)T},
\]
with $A,B,C$ as above.
\end{corollary}

\begin{proof}
This is Theorem~\ref{thm:forward} for $n=3$.  The displayed formulas for $A,B,C$ are obtained by inverting the Dickson matrix of $\mathcal L(X)=aX+bX^q+cX^{q^2}$.
\end{proof}

Similarly, we can also give an explicit criterion for the resulting permutation polynomials to be involutions.

\begin{corollary}\label{cor:trinomial-involution}
Keep the notation of Theorem~\ref{cor:trinomial}.  Suppose in addition that
\[
        r^2\equiv1\pmod {q^{3s}-1}
\]
and that the $q$-linearized trinomial
\[
        \mathcal L(X)=aX+bX^q+cX^{q^2}
\]
is an involution of $\F_{q^3}$.  Equivalently,
\[
        a^2+bc^q+cb^{q^2}=1,
\]
\[
        ab+ba^q+cc^{q^2}=0,
\]
and
\[
        ac+bb^q+ca^{q^2}=0.
\]
Then
\[
        F(x)=a x^r+b x^{r+T}+c x^{r+(q+1)T}
\]
is an involution of $K=\F_{q^{3s}}$.
\end{corollary}

\begin{proof}
  If $r^2\equiv1\pmod {q^{3s}-1}$, then $r'=r$.  Applying Theorem~\ref{thm:forward} to $\mathcal L^{-1}=\mathcal L$ gives $F^{-1}=F$.
\end{proof}

The following example gives a permutation trinomial that is an involution.

\begin{example}[A nonlinear trinomial involution over $\F_{2^{12}}$]\label{ex:trinomial-involution}
Let $q=4$, $s=2$, $K=\F_{2^{12}}$, and $T=2145$.  Let $\alpha\in\F_{64}$ be primitive such that
\[
        \alpha^6+\alpha+1=0.
\]
In $\F_{64}$ one checks that
\[
        a=\alpha,
        \qquad
        b=1,
        \qquad
        c=\alpha^{39}
\]
satisfy
\[
        a^2+bc^q+cb^{q^2}=1,
\]
\[
        ab+ba^q+cc^{q^2}=0,
\]
and
\[
        ac+bb^q+ca^{q^2}=0.
\]
Choose
\[
        r=1639.
\]
Then
\[
        r\equiv1\pmod {63},
        \qquad
        r^2\equiv1\pmod {4095}.
\]
By Corollary~\ref{cor:trinomial-involution},
\[
        J(x)=\alpha x^{1639}+x^{3784}+\alpha^{39}x^{79}
\]
is an involution of $\F_{2^{12}}$.  Hence its inverse is explicitly
\[
        J^{-1}(x)=J(x)=\alpha x^{1639}+x^{3784}+\alpha^{39}x^{79}.
\]
Again the exponents are not powers of $2$ modulo $4095$, so this is not a linearized example.
\end{example}

\section{Conclusion}

We presented a higher unit-torus lifting principle for sparse permutation polynomials and their compositional inverses.  The construction starts from a $q$-linearized permutation polynomial over $L=\F_{q^n}$ and lifts it to a sparse, generally nonlinear polynomial over $K=\F_{q^{ns}}$ by means of an exponent $T$ playing two simultaneous roles: $x^T$ lies in the norm-one torus $\mu_{(q^n-1)/(q-1)}$, while $T$ acts as $q-1$ on $L^{*}$.  The exponents
\[
e_i=\frac{q^i-1}{q-1}\qquad (0\leq i<n)
\]
are chosen so that they are compatible with the $q$-Frobenius action on the norm-one torus. Consequently, each nonzero term in the inverse of the underlying $q$-linearized polynomial gives rise to exactly one term in the inverse of the lifted nonlinear polynomial. Thus the lifting procedure preserves the number of nonzero terms in the inverse.

When $n=2$, the construction gives permutation binomials whose compositional inverses are again binomials, including explicit nonlinear examples such as
\[
x^{(q^4+q^3-q^2+q)/2}+ax
\]
over $\mathbb{F}_{q^4}$ for even $q$.  It also yields a simple involution criterion for this family.  When $n=3$, the construction produces permutation trinomials whose compositional inverses are again trinomials, together with explicit involutory examples.  More generally, the $n$-dimensional construction gives sparse permutation polynomials with at most $n$ terms whose compositional inverses also have at most $n$ terms.


A natural remaining problem is to understand the sparsity of the inverse of the underlying $q$-linearized polynomial.  In the present lifting construction, the number of terms in the lifted inverse is exactly controlled by the support of the inverse $q$-linearized polynomial over $\mathbb F_{q^n}$.  Thus it would be interesting to classify when this inverse has fewer than $n$ nonzero $q$-terms, or equivalently, to describe the vanishing patterns in the inverse Dickson matrix.


\end{document}